\documentclass[10pt]{article}
\usepackage{amsmath,amsthm,amssymb,amsfonts,amscd, array}
\usepackage{cite}
\usepackage{tikz}
\usepackage{url}
\usepackage[unicode,colorlinks]{hyperref}
\newtheorem{theorem}{Theorem}
\newtheorem{lemma}[theorem]{Lemma}

\newtheorem{example}{Example}
\newtheorem{remark}{Remark}

\begin{document}

\centerline{\bf \Large{On left legal semigroups}}
\medskip
\centerline{\bf Attila Nagy\footnote{email: nagyat@math.bme.hu}}
\medskip
\centerline{Department of Algebra}
\centerline{Budapest University of Technology and Economics}
\centerline{M\H uegyetem rkp. 3, Budapest, 1111, Hungary}

\bigskip

\noindent
\centerline{\bf Abstract}

\medskip

In this paper we study semigroups satisfying the identity $aba=ab$.

\section{Introduction}\label{sec1}
Many important combinatorial structures such as real and complex hyperplane arrangements, interval greedoids, matroids and oriented matroids have the structure of a left regular band (see, for example, \cite{Aguiar, Friedman, Margolis, Saliola}). A semigroup $S$ is called a left regular band if every element of $S$ is an idempotent element, and $S$ satisfies the identity $aba=ab$. Left regular bands are examined by many authors. See, for example, papers \cite{Auinger, Billhardt, Branco, Friedman, Saliola, Shevlyakov, Vakhrameev, Wang} and the book \cite{Petrich}. In \cite{Salam, Salam2}, the authors examined semigroups satisfying the identiy $aba=ab$. Using the terminology of \cite{Salam2}, such semigroups are called left legal semigroups. The class $\mathcal{LLS}$ of all left legal semigroups is a variety. Therefore, it is natural to find free semigroups and subdirectly irreducible semigroups in this variety. In Section~\ref{sec3}, we construct free left legal semigroups $F_{\mathcal{LLS}}(X)$ for arbitrary nonempty sets $X$. The subdirectly irreducible left legal semigroups are examined in Section~\ref{sec4}.
We show that a semigroup containing at least two elements is a subdirectly irreducible left legal semigroup satisfying the identity $ab=a^2b$ if and only if it is either a subdirectly irreducible left regular band or a two-element zero semigroup. We also show that a semigroup containing at least two elements is a subdirectly irreducible (left legal) semigroup satisfying the identity $ab=ac$ if and only if it is either a two-element left zero semigroup or a two-element zero semigroup.
The concept of the retract ideal extension of semigroups is used effectively in many cases in the study of the structure of semigroups. In Section~\ref{sec4}, we give a necessary and sufficient condition for a left legal semigroup to be a retract ideal extension of a left regular band by a zero semigroup.
One of the basic concepts of semigroup theory is the semilattice decomposition.
In Section~\ref{sec5}, this concept is in the centre of investigations.
We show that a semigroup $S$ is a semilattice indecomposable left legal semigroup if and only if $S^2$ is a left zero semigroup. The left (resp., right, weakly) separative semigroups are examined in many papers. In Section~\ref{sec6}, these semigroups are the focus. We show that a left legal semigroup is right separative if and only if it is weakly separative. It is also proved that every left legal left separative semigroup is commutative. In Section~\ref{sec7}, the lattice of all left legal semigroup varieties is examined.

\section{Preliminaries}\label{sec2}
By a \emph{semigroup} we mean a multiplicative semigroup, i.e. a nonempty set together with an associative multiplication.
A nonempty subset $I$ of a semigroup $S$ is said to be an \emph{ideal} of $S$ if $sa,as\in I$ for every $a\in I$ and $s\in S$.
If $I$ is an ideal of a semigroup $S$, then the relation $\varrho _I$ on $S$ defined by $(a, b)\in \varrho _I$ if and only if $a=b$ or $a, b\in I$ is a congruence on $S$, which is called the \emph{Rees congruence on $S$ determined by $I$}. The equivalence classes of $S$ mod $\varrho _I$ are $I$ itself and every one-element set $\{ a\}$ with $a\in S\setminus I$. The factor semigroup $S/\varrho_I$ is called the \emph{Rees factor semigroup of $S$ modulo $I$}. We shall write $S/I$ instead of $S/\varrho _I$. We may describe $S/I$ as the result of collapsing $I$ into a single (zero) element, while the elements of $S$ outside of $I$ retain their identity.
A semigroup $S$ containing a zero element $0$ is called a \emph{zero semigroup} if $ab=0$ for every $a, b\in S$. It is obvious that if $S$ is a semigroup, then $S^2$ is an ideal of $S$, and the Rees factor semigroup $S/S^2$ is a zero semigroup.
Let $B$ and $Q$ be disjoint semigroups, $Q$ having a zero element. A semigroup $S$ is called an \emph{ideal extension} of $B$ by $Q$ if $S$ contains $B$ as an ideal, and if the Rees factor semigroup $S/B$ is isomorphic with $Q$.
An ideal $I$ of a semigroup $S$ is called a \emph{retract ideal} if there is a homomorphism of $S$ onto $I$ which leaves the elements of $I$ fixed. Such a homomorphism is called a \emph{retract homomorphism} of $S$ onto $I$. In this case we say that $S$ is a \emph{retract ideal extension} of $I$.

\noindent
We say that a semigroup $S$ is a \emph{subdirect product} of semigroups $S_i$ ($i\in I$) if $S$ is isomorphic to a subsemigroup $T$ of the direct product of semigroups $S_i$ ($i\in I$) such that the restriction of the projection homomorphisms to $T$ are surjective. A semigroup $S$ is said to be \emph{subdirectly irreducible} if whenever $S$ is written as a subdirect product of a family of semigroups $S_i$ ($i\in I$), then, for at least one $j\in I$, the projection homomorphism $\pi _j$ maps $S$ onto $S_j$ isomorphically.
Subdirect decompositions of a semigroup $S$ are closely connected with congruences on $S$ (see, for example, \cite[I.3.6]{Petrich}).
If $\alpha _i$ ($i\in I$) are congruences on a semigroup $S$ and $\cap _{i\in I}\alpha _i=\iota_S$, the equality relation on $S$, then $S$ is a subdirect product of the factor semigroups $S/\alpha _i$. Conversely,
if a semigroup is a subdirect product of semigroups $S_i$ ($i\in I$) and $\alpha _i$ is the congruence on $S$ induced by the projection homomorphism $\pi _i$ ($i\in I$), then $\cap _{i\in I}\alpha _i=\iota_S$.

\noindent
If $\cal C$ is a class of semigroups, then a congruence $\varrho$ on a semigroup $S$ is called a $\cal C$-\emph{congruence} if the factor semigroup $S/\varrho$ belongs to $\cal C$.
An element $e$ of a semigroup is called an \emph{idempotent element} if $e^2=e$.
A semigroup in which every element is an idempotent element is called a \emph{band}.
A commutative band is called a \emph{semilattice}. It is clear that the universal relation on an arbitrary semigroup is a semilattice congruence. A semigroup $S$ is said to be \emph{semilattice indecomposable} if the universal relation is the only semilattice congruence on $S$. By Theorem of \cite{Tamura}, every semigroup $S$ has a least semilattice congruence $\eta _S$ whose classes are semilattice indecomposable semigroups. In other words, every semigroup is a semilattice of semilattice indecomposable semigroups.
A band satisfying the identity $aba=ab$ is called a \emph{left regular band}.
A semigroup satisfying the identity $ab=a$ is called a \emph{left zero semigroup}. By the dual of \cite[II.3.12. Proposition]{Petrich}, a band is left regular if and only if it is a semilattice of left zero semigroups. It is clear that every left zero semigroup is left legal and semilattice indecomposable.

\noindent
A semigroup $S$ is called a left legal semigroup if it satisfies the identity $aba=ab$. Every left regular band is a left legal semigroup. The converse is also true for semigroups $S$ containing a right identity element, because $a^2=aea=ae=a$ is satisfied for every element $a$ of $S$ and a right identity element $e$ of $S$.

\noindent
A semigroup $S$ is called an \emph{$\cal L$-commutative semigroup} \cite{NagyRC} if, for every elements $a, b\in S$, there is an element $x\in S^1$ such that $ab=xba$. A semigroup $S$ is called a \emph{right weakly commutative semigroup} if, for every $a, b\in S$, there exists $x\in S$ and a positive integer $n$ such that $(ab)^n=xa$.
It is clear that every $\cal L$-commutative semigroup is right weakly commutative.

\noindent
A semigroup $S$ is said to be a \emph{right separative semigroup} if $ab=b^2$ and $ba=a^2$ imply $a=b$ for every $a, b\in S$. A \emph{left separative semigroup} is defined analogously. A semigroup $S$ is called a \emph{weakly separative semigroup} if $a^2=ab=b^2$ implies $a=b$ for every $a, b\in S$.

\noindent
For an arbitrary semigroup $S$,
\[S^1=\begin{cases} S, &\textit{if $S$ has an identity element};\\
S\cup 1, &\textit{otherwise},\end{cases}\]
where $S\cup 1$ is the semigroup which is obtained by the adjunction of an identity element $1$ to $S$.

\noindent
For notions and notations not defined but used in this paper, we refer the reader to \cite{Clifford1}.

\section{Free left legal semigroups}\label{sec3}

Let $\mathcal V$ be a non-trivial variety of algebras, i.e. $\mathcal{V}$ contains algebras with more than one element. An algebra $F_{\mathcal {V}}(X)\in \mathcal{V}$ is said to be a free $\mathcal V$-algebra on the set $X$ if $X$ generates $F_{\mathcal{V}}(X)$, and every mapping of $X$ into any algebra $A$ from $\mathcal V$ can be extended to a homomorphism of $F_{\mathcal{V}}(X)$ into $A$. In this section we construct free left legal semigroups $F_{\mathcal{LLS}}(X)$ for arbitrary nonempty sets $X$.

\medskip

\noindent
First we prove two lemmas which will be used throughout the paper.

\begin{lemma}\label{lem0} If $S$ is a left legal semigroup, then every element of $S^2$ is an idempotent element. Moreover, $a^k=a^2$ for every $a\in S$ and every integer $k\geq 2$.
\end{lemma}

\begin{proof} Let $S$ be a left legal semigroup. For every elements $a$ and $b$ of $S$, we have
\[(ab)^2=a(bab)=aba=ab.\] Thus every element of $S^2$ is an idempotent element. For every $a\in S$, $a^3=a^2$ by definition, and hence $a^k=a^2$ for every integer $k\geq 2$.
\end{proof}

\begin{lemma}\label{lem3} For arbitrary elements $a$ and $b$ of a left legal semigroup, $ab=ab^2$.
\end{lemma}

\begin{proof} Using Lemma~\ref{lem0}, we get $ab=(ab)^2=(aba)b=ab^2$.
\end{proof}

\medskip

\noindent
Let $X$ be a nonempty set, whose elements are also said to be letters. A finite sequence of letters from $X$ is said to be a word. It is known \cite{Vakhrameev} that the free left regular band on $X$ consists of all words of the letters from $X$ which have no duplicate letters, and the multiplication $\cdot$ is defined by $\omega_1 \cdot \omega_2=\omega_1\exists_{\omega_1}(\omega_2)$, where $\exists_{\omega_1}(\omega_2)$ means that we remove from $\omega_2$ all letters contained in $\omega_1$.
As we will see below, the elements of the free left legal semigroup differ only in a small detail from the elements of the free left regular band. Before constructing the free left legal semigroup, we note that Lemma~\ref{lem0} and Lemma~\ref{lem3} imply that
every product of elements of a left legal semigroup can be reduced to a product of the form $a^k_1a_2\cdots a_n$, equal to the original product, in which the elements $a_1\dots a_n$ are pairwise distinct and $k=1$ or $k=2$.
Accordingly, for a nonempty set $X$, let $F_{\mathcal{LLS}}(X)$ denote the set of all nonempty words of the letters from $X$ in which each letter $x\in X$ can occur at most once or twice, in the latter case in the form $xx$ at the beginning of the word.
For an arbitrary word $\omega \in F_{\mathcal{LLS}}(X)$, let $\omega ^*=\omega$, if the letters of $\omega$ are pairwise distinct, and let $\omega ^*$ be the word which can be obtained from $\omega$ by deleting one of its first two letters if they are equal to each other.
We define an operation $\circ$ on $F_{\mathcal{LLS}}(X)$: for arbitrary words $\omega _1=x_1 \dots x_n$ and $\omega _2=y_1 \dots y_m$ of $F_{\mathcal{LLS}}(X)$, let
\begin{equation}\label{operation}
\omega _1\circ \omega _2=\begin{cases}
xx\exists_{\omega_1}(\omega ^*_2), & \textit{if $\omega _1=x=y_1$ for some $x\in X$,}\\
\omega _1\exists _{\omega _1}(\omega^* _2), & \textit{otherwise.}\end{cases}
\end{equation}

\begin{example}\label{ex1}
If $X=\{ x, y\}$, then
$F_{\mathcal{LLS}}(X)=\{x, xx, y, yy, xy, xxy, yx, yyx\}$. The operation $\circ$ on $F_{\mathcal{LLS}}(X)$ is given by Table~\ref{Table 1}.

\medskip

\begin{table}[htbp]
\begin{center}
\begin{tabular}{l|l l l l l l l l}
 $\circ$ &$x$&$xx$&$y$&$yy$&$xy$&$xxy$&$yx$&$yyx$\\
\hline
$x$&$xx$&$xx$&$xy$&$xy$&$xxy$&$xxy$&$xy$&$xy$\\
$xx$&$xx$&$xx$&$xxy$&$xxy$&$xxy$&$xxy$&$xxy$&$xxy$\\
$y$&$yx$&$yx$&$yy$&$yy$&$yx$&$yx$&$yyx$&$yyx$\\
$yy$&$yyx$&$yyx$&$yy$&$yy$&$yyx$&$yyx$&$yyx$&$yyx$\\
$xy$&$xy$&$xy$&$xy$&$xy$&$xy$&$xy$&$xy$&$xy$\\
$xxy$&$xxy$&$xxy$&$xxy$&$xxy$&$xxy$&$xxy$&$xxy$&$xxy$\\
$yx$&$yx$&$yx$&$yx$&$yx$&$yx$&$yx$&$yx$&$yx$\\
$yyx$&$yyx$&$yyx$&$yyx$&$yyx$&$yyx$&$yyx$&$yyx$&$yyx$
\end{tabular}
\caption[]{}\label{Table 1}
\end{center}
\end{table}
\end{example}

\begin{theorem}\label{thmleftlegal} If $X$ is an arbitrary nonempty set, then $(F_{\mathcal{LLS}}(X); \circ)$ is the free left legal semigroup on the set $X$.
\end{theorem}

\begin{proof} Let $X$ be a nonempty set and $F_X$ be the free semigroup on $X$, i.e. the set of all arbitrary nonempty words of letters from $X$.
Let $I$ denote the set of identities $uvu=uv$, $u^3=u^2$, and $uv^2=uv$. By Lemma~\ref{lem0} and Lemma~\ref{lem3}, a semigroup is a left legal semigroup if and only if it satisfies the identities belonging to $I$.
By \cite[Section V.]{Evans}, the free left legal semigroup is obtained as a quotient semigroup $F_X/\theta_I$, where the congruence $\theta_I$ is constructed on $F_X$ by defining two words to be $\theta_I$-congruent if they are equal or we can transform one into the other by a finite sequence of $I$-transformations, i.e. applications of the identities $uvu=uv$, $u^3=u^2$, and $uv^2=uv$.
The factor semigroup $F_X/\theta_I$ belongs to the variety $\mathcal{LLS}$ of all left legal semigroups and freely generated in $\mathcal{LLS}$ by the elements $[x]$ ($x\in X$), where $[x]$ denotes the $\theta_I$-class of $F_X$ containing the letter $x$.
We show that the groupoid $(F_{\mathcal{LLS}}(X); \circ)$ is an isomorphic image of the factor semigroup $F_X/\theta_I$, from which the statement of the theorem follows.
Let $(F_X, \rightarrow )$ be a length reducing rewrite system where $\rightarrow$ sends a word $\omega\in F_X$ to a shorter word of $F_X$ which can be obtained from $\omega$ by applying one of the identities in $I$. For example, $\rightarrow$ send the word $xxxyz$ to the word $xxyz$ by applying the identity $u^3=u^2$.
A word $\omega\in F_X$ is said to be a normal form if there is no $\hat{\omega}\in F_X$ with $\omega \rightarrow \hat{\omega}$. Let $\xrightarrow{*}$ is the reflexive-transitive closure of $\rightarrow$, that is $\omega\xrightarrow{*} \tilde{\omega}$ means that $\omega=\tilde{\omega}$ or there are words $\omega_1, \dots , \omega_n\in F_X$ such that $\omega=\omega_1\rightarrow \omega_2\rightarrow \dots \rightarrow \omega_n=\tilde{\omega}$.
A word $\omega'\in F_X$ is said to be a normal form of a word $\omega\in F_X$ if $\omega'$ is a normal form and $\omega \xrightarrow{*} \omega'$. For a word $\omega\in F_X$, let $v(\omega)=z_1z_2\dots z_k$ be the word of $F_X$ which satisfies the following conditions: $\{z_1, z_2, \dots , z_k\}$ is the set of all letters from $X$ belonging to $\omega$, and $z_{i-1}$ appears in $\omega$ sooner than $z_i$ ($i=2, \dots , k$). Obviously, $z_1$ equals the first letter of $\omega$ for arbitrary $\omega\in F_X$. If the first two letters of a word $\omega\in F_X$ of length at least two are not equal, then $z_2$ equals the second letter of $\omega$. It is easy to see that every word $\omega=x_1x_2\dots x_n$ has one and only one normal form, denoted by $\omega'$: if $n=1$ then $\omega'=\omega$; if $n\geq 2$, then $\omega'$
can be obtained by the help of $v(\omega)=z_1z_2\dots z_k$ as follows
\begin{equation}\label{normalforms}
\omega'=\begin{cases} x_1z_2\dots z_k, &\textit{if $x_1\neq x_2$}\\
x_1x_1z_2\dots z_k, &\textit{if $x_1=x_2$.}\end{cases}
\end{equation}
From this result we get that the set of all normal forms of $F_X$ equals $F_{\mathcal{LLS}}(X)$.
From (\ref{normalforms}) it follows that $\omega'_1=\omega'_2$ is satisfied for words $\omega_1=x_1x_2\dots x_n$ and $\omega_2=y_1y_2\dots y_m$ ($n, m\geq 2$) of $F_X$ if and only if the prefixes $x_1x_2$ and $y_1y_2$ are equal, and $v(\omega_1)=v(\omega_2)$.
It is also true that words $\omega_1=x_1x_2\dots x_n$ and $\omega_2=y_1y_2\dots y_m$ ($n, m\geq 2$) are $\theta_I$-congruent if and only if both of $x_1x_2=y_1y_2$ and $v(\omega_1)=v(\omega_2)$ are satisfied. Thus every $\theta_I$-class of $F_X$ contains one and only one normal form, also considering that $[x]=\{ x\}$ for every $x\in X$. Since $F_{\mathcal{LLS}}(X)$ is the set of all normal forms, $\Phi : [\omega]\to \omega'$ is a bijective mapping of the factor set $F_X/\theta_I$ onto the set $F_{\mathcal{LLS}}(X)$, where $\omega'$ denotes the normal form of $\omega$.
It is easy to see that, for arbitrary words $\omega_1$ and $\omega_2=y_1\dots y_m$ of $F_X$,
\begin{equation}\label{third}
(\omega_1\omega_2)'=\begin{cases}xx\omega'_1\exists_{\omega'_1}((\omega'_2)^*), &\textit{if $\omega_1=x=y_1$ for some $x\in X$,}\\
\omega'_1\exists_{\omega'_1}((\omega'_2)^*), &\textit{otherwise.}\end{cases}
\end{equation}
Using (\ref{operation}) and (\ref{third}), we have $(\omega_1\omega_2)'=\omega'_1\circ \omega'_2$. Then
\[\Phi([\omega_1]\bullet [\omega_2])=\Phi([\omega_1\omega_2])=(\omega_1\omega_2)'=
\omega'_1\circ \omega'_2=\Phi([\omega_1])\circ \Phi([\omega_2]).\]
Thus $\Phi$ is a homomorphism. Consequently $\Phi$ is an isomorphism of the semigroup $F_X/\theta_I$ onto the groupoid $(F_{\mathcal{LLS}}(X); \circ )$. Hence $(F_{\mathcal{LLS}}(X); \circ )$ is a semigroup, which is isomorphic to the free left legal semigroup $(F_X/\theta_I; \bullet )$. The set $X$ is the free generating system of $(F_{\mathcal{LLS}}(X); \circ )$.
\end{proof}

\section{Retract ideal extensions of left regular bands by zero semigroups}\label{sec4}

By Lemma~\ref{lem0}, if $S$ is a left legal semigroup, then every element of $S^2$ is an idempotent element. Thus every left legal semigroup is an ideal extension of a left regular band by a zero semigroup. In this section we focus on the case when this extension is retract. We note that if a semigroup $S$ is an ideal extension of a band $B$ by a zero semigroup, then $B=S^2$, because $B\subseteq S^2$ and the assumption that the Rees factor semigroup $S/B$ is a zero semigroup implies $S^2\subseteq B$.

\begin{theorem}\label{thm1} The following three conditions on a semigroup $S$ are equivalent.
\begin{itemize}
\item[(i)] $S$ is a retract ideal extension of a left regular band by a zero semigroup.
\item[(ii)] $S$ is a subdirect product of a left regular band and a zero semigroup.
\item[(iii)] $S$ is a left legal semigroup satisfying the identity $ab=a^2b$.
\end{itemize}
\end{theorem}

\begin{proof}

\noindent
$(i)\Rightarrow (ii)$: Assume that $S$ is a retract ideal extension of a left regular band $B$ by a zero semigroup. Then $B=S^2$ is a retract ideal of $S$.
Let $\varphi$ be a retract homomorphism of $S$ onto $B$. It is easy to see that\[ker_{\varphi}\cap \varrho _{B}=\iota _S,\] where $ker_{\varphi}$ denotes the kernel of $\varphi$, and $\varrho _{B}$ denotes the Rees congruence on $S$ modulo the ideal $B$. Then, by \cite[I.3.6]{Petrich}, $S$ is a subdirect product of a left regular band and a zero semigroup, because the factor semigroup $S/ker_{\varphi}$ is isomorphic to the left regular band $B$, and the Rees factor semigroup $S/B$ is a zero semigroup.

\medskip

\noindent
$(ii)\Rightarrow (iii)$: It is obvious, because every subsemigroup of the direct product of a left regular band and a zero semigroup is a left legal semigroup satisfying the identity $ab=a^2b$.

\medskip

\noindent
$(iii)\Rightarrow (i)$:  Assume that $S$ is a left legal semigroup satisfying the identity $ab=a^2b$. By Lemma~\ref{lem0},
$S$ is an ideal extension of the left regular band $S^2$ by the zero semigroup $S/S^2$. Let $\varphi$ be the mapping of $S$ into $S^2$ defined by the following way: for an arbitrary $a\in S$, $\varphi (a)=a^2$. As $S^2$ is a band, $\varphi$ maps $S$ onto $S^2$ and leaves the elements of $S^2$ fixed. Let $a, b\in S$ be arbitrary elements. Using Lemma~\ref{lem0}, Lemma~\ref{lem3} and the assumption that $S$ satisfies the identity $ab=a^2b$, we have
\[\varphi(a)\varphi(b)=a^2b^2=a^2b=ab=(ab)^2=\varphi(ab).\]
Then $\varphi$ is a retract homomorphism of $S$ onto $S^2$. Consequently $S$ is a retract ideal extension of the left regular band $S^2$ by the zero semigroup $S/S^2$.
\end{proof}

\begin{example}\label{ex2}

On the set $S=\{x, e, f, g, h\}$, consider the operation defined by Table~\ref{Table 3}.

\begin{table}[htbp]
\begin{center}
\begin{tabular}{l|l l l l l}
 &$x$&$e$&$f$&$g$&$h$\\
\hline
$x$&$e$&$e$&$e$&$g$&$h$\\
$e$&$e$&$e$&$e$&$g$&$h$\\
$f$&$f$&$f$&$f$&$g$&$h$\\
$g$&$g$&$g$&$g$&$g$&$g$\\
$h$&$h$&$h$&$h$&$h$&$h$
\end{tabular}
\caption[]{}\label{Table 3}
\end{center}
\end{table}

\medskip

\noindent
It is a matter of checking to see that this operation is associative. $S^2=\{ e, f, g, h\}$ is a semilattice of the left zero semigroups $S_1=\{ e, f\}$ and $S_2=\{ g, h\}$, and hence $S^2$ is a left regular band. The mapping $\varphi : s\mapsto s^2$ is a retract homomorphism of $S$ onto $S^2$.
Then, by Theorem~\ref{thm1},
$S$ is a subdirect product of the left regular band $S^2$ and the two-element zero semigroup $S/S^2=\{x, 0\}$.
\end{example}

\medskip

\noindent
In the free left legal semigroup $F_{\mathcal{LLS}}(X)$ considered in Example~\ref{ex1}, $xy\neq xxy$. Thus not every left legal semigroup satisfies the conditions in Theorem~\ref{thm1}.

\begin{theorem} A semigroup containing at least two elements is a subdirectly irreducible left legal semigroup satisfying the identity $ab=a^2b$ if and only if it is either a subdirectly irreducible left regular band or a two-element zero semigroup.
\end{theorem}

\begin{proof} A left regular band is a left legal semigroup and satisfies the identity $ab=a^2b$. It is obvious that a two-element zero semigroup is a subdirectly irreducible left legal semigroup satisfying the identity $ab=a^2b$.
To prove the converse assertion, assume that $S$ is a subdirectly irreducible left legal semigroup which has at least two elements and satisfies the identity $ab=a^2b$. Then, by $(ii)$ of Theorem~\ref{thm1}, $S$ is either a subdirectly irreducible left regular band or a subdirectly irreducible zero semigroup $Z$. Since a zero semigroup is subdirectly irreducible if and only if it has at most two elements, we have $\mid Z\mid=2$.
\end{proof}

\noindent
For subdirectly irreducible left regular bands, we refer the reader to \cite[Theorem 2.16]{Wang}.

\begin{theorem}\label{thm2} The following three conditions on a semigroup $S$ are equivalent.
\begin{itemize}

\item[(i)] $S$ is a retract ideal extension of a left zero semigroup by a zero semigroup.
\item[(ii)] $S$ is a subdirect product of a left zero semigroup and a zero semigroup.
\item[(iii)] $S$ satisfies the identity $ab=ac$ (and so $S$ is a left legal semigroup).
\end{itemize}
\end{theorem}

\begin{proof}
$(i)\Rightarrow (ii)$:  Assume that $S$ is a retract ideal extension of a left zero semigroup $L$ by a zero semigroup. Then $L$ is an ideal of $S$, and the Rees factor semigroup $S/L$ is a zero semigroup.
Let $\varphi$ be a retract homomorphism of $S$ onto $L$. Then \[ker_{\varphi}\cap \varrho _{L}=\iota _S,\] and hence \cite[I.3.6]{Petrich} implies that $S$ is a subdirect product of the left zero semigroup $L\cong S/ker_{\varphi}$ and the zero semigroup $S/L$.

\medskip

\noindent
$(ii)\Rightarrow (iii)$: It is obvious, because every subsemigroup of the direct product of a left zero semigroup and a zero semigroup satisfies the identity $ab=ac$.

\medskip

\noindent
$(iii)\Rightarrow (i)$: Let $S$ be a semigroup satisfying the identity
$ab=ac$. Then $S$ is a left legal semigroup. By Lemma~\ref{lem0}, $S^2$ is a band.
Then, for arbitrary elements $e, f\in S^2$, we have \[ef=e^2=e,\] using the assumption that $S$ satisfies the identity $ab=ac$. Hence $S^2$ is a left zero semigroup.
Let $\varphi$ be the mapping of $S$ into $S^2$ defined by the following way: for an arbitrary $a\in S$, $\varphi (a)=a^2$. As $S^2$ is a band, $\varphi$ maps $S$ onto $S^2$ which leaves the elements of $S^2$ fixed. For arbitrary  $a, b\in S$,
\[\varphi(a)\varphi(b)=a^2b^2=a(abb)=a(bab)=(ab)^2=\varphi(ab),\] using again the assumption that $S$ satisfies the identity $ax=ay$.
Thus $\varphi$ is a homomorphism. Consequently $S$ is a retract ideal extension of the left zero semigroup $S^2$ by the zero semigroup $S/S^2$.
\end{proof}

\begin{example}\rm\label{ex3}
On the set $S=\{x, y, b, c\}$, consider the operation defined by Table~\ref{Table 4}.

\begin{table}[htbp]
\begin{center}
\begin{tabular}{l|l l l l}
 &$x$&$y$&$b$&$c$\\
\hline
$x$&$b$&$b$&$b$&$b$\\
$y$&$c$&$c$&$c$&$c$\\
$b$&$b$&$b$&$b$&$b$\\
$c$&$c$&$c$&$c$&$c$
\end{tabular}
\caption[]{}\label{Table 4}
\end{center}
\end{table}

\medskip

\noindent
It is easy to check that $S$ is a semigroup with respect to this operation, and $S$ is an ideal extension of the left zero semigroup $S^2=\{b, c\}$ by the zero semigroup $S/S^2=\{x, y, 0\}$. Moreover, the mapping $\varphi \colon s\mapsto s^2$ is a retract homomorphism of $S$ onto $S^2$. Thus the semigroup $S$ satisfies the conditions of Theorem~\ref{thm2}.
\end{example}

\medskip

\noindent
The next example shows that if a semigroup is an ideal extension of a left zero semigroup by a zero semigroup, then it does not necessarily satisfy the conditions of Theorem~\ref{thm2}.

\begin{example}\label{ex3a}\rm
On the set $S=\{a, b, c, d, e, f\}$, consider the operation defined by Table~\ref{Table 5}.

\begin{table}[ht]
\begin{center}
\begin{tabular}{l|l l l l l l}
 &$a$&$b$&$c$&$d$&$e$&$f$\\
\hline
$a$&$e$&$c$&$e$&$c$&$e$&$c$\\
$b$&$d$&$f$&$d$&$f$&$d$&$f$\\
$c$&$c$&$c$&$c$&$c$&$c$&$c$\\
$d$&$d$&$d$&$d$&$d$&$d$&$d$\\
$e$&$e$&$e$&$e$&$e$&$e$&$e$\\
$f$&$f$&$f$&$f$&$f$&$f$&$f$
 \end{tabular}
\caption[]{}\label{Table 5}
\end{center}
\end{table}

\medskip

\noindent
It is easy to see that $S$ is a semigroup which is an ideal extension of the left zero semigroup $S^2=\{c, d, e, f\}$ by the zero semigroup $S/S^2=\{a, b, 0\}$. Since $ab\neq ac$, the semigroup $S$ does not satisfy the conditions of Theorem~\ref{thm2}.
\end{example}

\begin{theorem} A semigroup containing at least two elements is a subdirectly irreducible
semigroup satisfying the identity $ab=ac$ if and only if it is either a two-element left zero semigroup or a two-element zero semigroup.
\end{theorem}

\begin{proof} It is obvious that every two-element semigroup is subdirectly irreducible. Moreover, every left zero semigroup and every zero semigroup satisfies the identity $ab=ac$.
To prove the converse assertion, let $S$ be a subdirectly irreducible
semigroup which has at least two elements and satisfies the identity $ab=ac$. Then, by $(ii)$ of Theorem~\ref{thm2}, $S$ is either a left zero semigroup or a zero semigroup. Since a left zero semigroup and, similarly, a zero semigroup is subdirectly irreducible if and only if it contains at most two elements, $S$ is either a two-element left zero semigroup or a two-element zero semigroup.
\end{proof}

\section{The least semilattice congruence on left legal semigroups}\label{sec5}
In the study of the structure of semigroups, the semilattice decomposition of semigroups plays a basic role.
It is proved in \cite{Tamura} that every semigroup $S$ has a least semilattice congruence $\eta_S$ whose classes are semilattice indecomposable semigroups; the induced partition of $S$ is called a semilattice decomposition of $S$, and the congruence classes of $\eta_S$ are said to be the semilattice components of $S$.

\noindent
A semigroup $S$ is said to be an \emph{archimedean semigroup} if, for arbitrary $a, b\in S$, there is a positive integer $n$ such that $a^n\in SbS$ and $b^n\in SaS$. Every archimedean semigroup is semilattice indecomposable. It is proved in \cite{Putcha} that a semigroup is a semilattice of archimedean semigroups if and only if it is a Putcha semigroup. Recall that a semigroup $S$ is called a \emph{Putcha semigroup} if, for every $a, b\in S$, the assumption $a\in S^1bS^1$ implies $a^m\in S^1b^2S^1$ for some positive integer $m$.

\begin{theorem}\label{thm3} Every left legal semigroup is a Putcha semigroup.
\end{theorem}

\begin{proof} Let $S$ be a left legal semigroup. It is clear that $S$ is $\cal L$-commutative, and hence right weakly commutative. By \cite[Theorem 4.1]{Nagybook} and \cite[Lemma 2.2]{Nagybook}, every right weakly commutative semigroup is a Putcha semigroup, which proves our assertion.
\end{proof}

\begin{remark}\label{rm1}
Theorem~\ref{thm3} and \cite[Theorem 2.1]{Nagybook} together imply that if $S$ is a left legal semigroup, then,
\[\eta _S=\{(a,b)\in S\times S:\ a^n\in SbS, \ b^n\in SaS \ \hbox{for a positive integer}\ n\},\]
where $\eta _S$ denotes the least semilattice congruence on $S$. This implies that a left legal semigroup is semilattice indecomposable if and only if it is archimedean.
\end{remark}

For an arbitrary semigroup $S$, let $\tau _S$
denote the binary relations on $S$ defined by the following way: $(a, b)\in \tau _S$
for elements $a, b\in S$, if and only if there is a positive integer $n$ such that $a^nb=a^{n+1}$ and $b^na=b^{n+1}$.
By \cite[Lemma 4.1]{Nagybook}, $\tau _S$ and its dual $\sigma _S$ are equivalence relations on an arbitrary semigroup $S$. Equivalence relations
$\tau _S$ and $\sigma _S$ are studied by many authors in special classes of semigroups. In the book \cite{Clifford1}, Theorem 4.14 asserts that if $S$ is a commutative semigroup, then $\tau _S=\sigma _S$ is a least weakly separative congruence on $S$. In \cite[Theorem 5.4]{Chrislock}, Chrislock shown that if $S$ is a medial semigroup, then  $\tau _S$ is the least left separative congruence on $S$, and $\sigma _S$ is the least right separative congruence on $S$. In \cite[Theorem 6]{Mukherjee}, Mukherjee proved that $\sigma _S$ is a weakly separative congruence on a quasi-commutative semigroup $S$. It is also conjectured that $\sigma _S$ is the least weakly separative congruence. Pondelicek proved \cite[Theorem 5]{Pondelicek} that if $S$ is a duo semigroup, then $\tau _S\cap \sigma _S$ is the least weakly separative congruence on $S$. Since every quasicommutative semigroup is a duo semigroup, this result implies that $\tau _S\cap \sigma _S$ is the least weakly separative congruence on a quasi-commutative semigroup, giving a negative answer for Mukherjee's conjecture. Generalizing and completing the above results, in \cite[Theorem 1]{Nagy} it is proved that $\tau _S\cap \sigma _S$ is the least weakly separative congruence on a weakly commutative semigroup $S$. Using this result, in \cite{Nagy2} the weakly separative weakly commutative semigroups are characterized.

\medskip

\noindent
Since every left legal semigroup is right weakly commutative, \cite[Theorem 4.7]{Nagybook} implies that $\tau _S$ is a weakly separative congruence on a left legal semigroup. The next theorem gives an additional result about $\tau _S$ if $S$ is a left legal semigroup.

\begin{theorem}\label{thm4} If $S$ is a left legal semigroup, then $\tau _S=\eta_S$, where $\eta_S$ is the least semilattice congruence on $S$.
\end{theorem}

\begin{proof} We use Remark~\ref{rm1}.
First we show that $\eta _S\subseteq \tau _S$. Assume $(a, b)\in \eta _S$ for elements $a, b\in S$. Then, by Remark~\ref{rm1}, $a^n=xby$ and $b^n=uav$ for a positive integer $n$ and for some elements $x, y, u, v\in S$. We can suppose $n\geq 2$. Then
\[a^nb=(xby)b=x(byb)=xby=a^n=a^{n+1},\] because $a^k=a^2$ for every integer $k\geq 2$ by Lemma~\ref{lem0}. Similarly, $b^na=b^{n+1}$.
Thus $(a, b)\in  \tau _S$. Hence $\eta _S \subseteq \tau _S$. To show $\tau _S\subseteq \eta _S$, assume $(a, b)\in \tau _S$ for elements $a, b\in S$. Then
$a^nb=a^{n+1}$ and $b^na=b^{n+1}$ for some positive integer $n\geq 2$. Thus \[a^{n+2}=a^nba\in SbS\quad \hbox{and}\quad b^{n+2}=b^nab\in SaS.\] Hence $(a, b)\in \eta _S$. Thus $\tau _S\subseteq \eta _S$. Consequently $\eta _S =\tau _S$.
\end{proof}

\begin{remark}\label{remyy} By Lemma~\ref{lem0}, if $a$ is an arbitrary element of a left legal semigroup, then $a^k=a^2$ is satisfied
for every integer $k\geq 2$. Thus $(a, b)\in \tau _S$ is satisfied for elements $a$ and $b$ of a left legal semigroup if and only if $a^2b=a^2$ and $b^2a=b^2$.
\end{remark}

\begin{theorem}\label{thm5} The following three conditions on a semigroup $S$ are equivalent.
\begin{itemize}
\item[(i)] $S$ is a semilattice indecomposable left legal semigroup.
\item[(ii)] $S$ is a left legal semigroup satisfying the identity $a^2b=a^2$.
\item[(iii)] $S^2$ is a left zero semigroup.
\end{itemize}
\end{theorem}

\begin{proof} By Theorem~\ref{thm4} and Remark~\ref{remyy}, a left legal semigroup $S$ is semilattice indecomposable if and only if $\tau _S$ is the identity relation on $S$, i.e. $a^2b=a^2$ is satisfied for every elements $a$ and $b$ of $S$. Thus it is sufficient to show that $(ii)$ and $(iii)$ are equivalent.

\medskip

\noindent
$(ii)\Rightarrow (iii)$: Assume that $S$ is a left legal semigroup satisfying the identity $a^2b=a^2$. Since $S$ is left legal, $S^2$ is a band by Lemma~\ref{lem0}. Let $e, f\in S^2$ be arbitrary elements. Then \[ef=e^2f=e^2=e,\] i.e. $S^2$ is a left zero semigroup.

\medskip

\noindent
$(iii)\Rightarrow (ii)$: Assume that $S^2$ is a left zero semigroup. Let $a, b\in S$ be arbitrary elements. Then \[aba=(ab)(ab)a=(ab)(aba)=ab.\] Thus $S$ is a left legal semigroup. Moreover, \[a^2b=(a^2a^2)b=a^2(a^2b)=a^2.\]
\end{proof}

\medskip

\noindent
The semigroup $S$ defined in Example~\ref{ex3} is a semilattice
indecomposable left legal semigroup in which $S^2$ is a left zero semigroup such that it is a retract ideal of $S$.
The semigroup $S$ defined in Example~\ref{ex3a} is also a semilattice indecomposable left legal semigroup in which $S^2$ is a left zero semigroup, but $S^2$ is not a retract ideal of $S$.

\medskip

\noindent
By \cite[II. 1.4. Lemma]{Petrich} and the dual of \cite[II. 3. 12. Proposition]{Petrich}, a band $B$ is a left regular band if and only if the classes of its least semilattice congruence $\eta_B$ are left zero semigroups. Moreover, $(e, f)\in \eta_B$ is satisfied for elements $e, f\in B$ if and only if $ef=e$ and $fe=f$. The next theorem shows how to get the classes of the least semilattice congruence on a left legal semigroup.

\begin{theorem}\label{thm6a}
Let $S$ be a left legal semigroup and $L_i$ ($i\in I$) be left zero subsemigroups of $S$ such that the left regular band $S^2$ is a semilattice of $L_i$ ($i\in I$). Then
the classes of the least semilattice congruence $\eta_S$ on $S$ are $S_i=\{a\in S:\ a^2\in L_i\}$ ($i\in I$).
\end{theorem}

\begin{proof} Let $a, b\in S$ be arbitrary elements. By Theorem~\ref{thm4} and Remark~\ref{remyy}, $(a, b)\in \eta_S$ if and only if $(a, b)\in \tau_S$, i.e.
\[a^2b=a^2\quad \hbox{and}\quad b^2a=b^2.\] By Lemma~\ref{lem3}, it is equivalent to the condition
\[a^2b^2=a^2\quad \hbox{and}\quad b^2a^2=b^2,\] i.e. $a^2, b^2\in L_i$ for some $i\in I$, which is equivalent to the condition that $a, b\in S_i$. Thus the classes of the least semilattice congruence $\eta_S$ on $S$ are $S_i$ ($i\in I$).
\end{proof}

\medskip
\noindent

From Theorem~\ref{thm6a} it follows that the maximal semilattice image of a left legal semigroup $S$ is equal to that of its maximal sub-left regular band $S^2$, which in turn is isomorphic to the semilattice of its $\mathcal L$-classes = $\mathcal J$-classes.

\medskip
\noindent
It is easy to see that $\{xx\}$, $\{ yy\}$, and $\{xy, xxy, yx, yyx\}$ are left zero subsemigroups of the free left legal semigroup $F_{\mathcal{LLS}}(X)$ considered in Example~\ref{ex1} such that $F^2_X=\{xx, yy, xy, xxy, yx, yyx\}$ is a semilattice of them.  By Theorem~\ref{thm6a}, the classes of the least semilattice congruence on $F_{\mathcal{LLS}}(X)$ are $A=\{x, xx\}$, $B=\{y, yy\}$, and  $C=\{xy, xxy, yx, yyx\}$.
It can also be directly proved that $A$, $B$, and $C$ are semilattice indecomposable subsemigroups of $F_{\mathcal{LLS}}(X)$, and the equivalence relation on $F_{\mathcal{LLS}}(X)$ whose classes are $A$, $B$, and $C$ is a semilattice congruence on $F_{\mathcal{LLS}}(X)$.

\section{Right and left separativity of left legal semigroups}\label{sec6}

\begin{theorem}\label{thm7} The following two conditions on a left legal semigroup $S$ are equivalent.
\begin{itemize}
\item[(i)] $S$ is right separative.
\item[(ii)] $S$ is weakly separative.
\end{itemize}
\end{theorem}

\begin{proof} $(i)\Rightarrow (ii)$: Let $S$ be a right separative left legal semigroup. To show that $S$ is weakly separative, assume $a^2=ab=b^2$ for elements $a, b\in S$. Using Lemma~\ref{lem0}, $ab=b^2$ implies \[bab=b^3=b^2,\] and hence \[ba=bab=b^2=a^2.\] Then
\[ab=b^2\quad \hbox{and}\quad ba=a^2.\] Since $S$ is right separative, we get $a=b$. Consequently $S$ is weakly separative.

\medskip

\noindent
$(ii)\Rightarrow (i)$: Let $S$ be a weakly separative left legal semigroup. To show that $S$ is right separative, assume
$ab=b^2$ and $ba=a^2$ for elements $a, b\in S$. Using the first equation and Lemma~\ref{lem0}, we get \[ba=bab=b^3=b^2,\] and hence \[b^2=ba=a^2.\] Since $S$ is weakly separative, we get $a=b$. Consequently $S$ is right separative.
\end{proof}

\begin{remark}\label{rm2} It is clear that every left zero semigroup containing at least two elements is left legal and weakly separative, but not left separative. Thus a weakly separative left legal semigroup is not left separative, in general.
\end{remark}

\begin{theorem}\label{thm8} Every left separative left legal semigroup is commutative.
\end{theorem}

\begin{proof} Let $S$ be a left separative left legal semigroup. Then, for every $a, b\in S$,
\[(ab)(ba)=(aba)(bab)=(ab)^3=(ab)^2,\] using Lemma~\ref{lem0}. Similarly, \[(ba)(ab)=(ba)^2.\] Since $S$ is left separative, we get $ab=ba$, and hence $S$ is commutative.
\end{proof}

\section{On the lattice of all left legal semigroup varieties}\label{sec7}

The class of all varieties of semigroups forms a lattice under the following operations: for varieties $\mathcal X$ and $\mathcal Y$, their join ${\mathcal X}\vee {\mathcal Y}$ is the variety
generated by the set-theoretical union of ${\mathcal X}$ and $\mathcal Y$ (as classes of semigroups),
and their meet $\mathcal{X} \wedge \mathcal{Y}$ coincides with the set-theoretical intersection of $\mathcal X$ and $\mathcal Y$.
A variety $\mathcal V$ of semigroups is said to be a left legal semigroup variety, if every semigroup belonging to $\mathcal V$ is left legal. The set of all left legal semigroup varieties form a sublattice of the lattice of all semigroup varieties. This sublattice denoted by ${\mathcal L}({\bf LLS})$.
In this section we use the following notations:

\begin{itemize}
\item[] $\mathcal {A}$: the variety of semigroups satisfying the identities $aba=ab$, $ab=a^2b$;
\item[] $\mathcal {B}$: the variety of semigroups satisfying the identity $ab=ac$;
\item[] $\mathcal{C}$: the variety of semigroups satisfying the identities $abc=acb$, $ab=a^2b$;
\item[] $\mathcal{D}$: the variety of semigroups satisfying the identities $ab=ba$, $ab=a^2b$;
\item[] $\mathcal {LRB}$: the variety of all left regular bans ($a^2=a$, $aba=ab$);
\item[] $\mathcal {LNB}$: the variety of all left normal bands ($a^2=a$ $abc=acb$);
\item[] $\mathcal {SL}$: the variety of all semilattices ($a^2=a$, $ab=ba$);
\item[] $\mathcal {LZ}$: the variety of all left zero semigroups ($ab=a$);
\item[] $\mathcal {ZM}$: the variety of all zero semigroups ($ab=0$);
\item[] $\mathcal {T}$: the trivial variety ($a=b$).
\end{itemize}

\medskip

\noindent
We note that if a semigroup $S$ belongs to the variety $\mathcal{C}$, then $aba=a^2b=ab$ for arbitrary $a, b\in S$, and hence $S$ is a left legal semigroup. Then, by Lemma~\ref{lem3}, $ab=ab^2$ is also satisfied for arbitrary $a, b\in S$. Thus a semigroup belongs to the variety $\mathcal{C}$ if and only if it satisfies the identities $abc=acb$, $ab=a^2b$, and  $ab=ab^2$.

\medskip
\noindent
It is easy to see that \[\mathcal{LZ}, \mathcal{ZM}\subseteq \mathcal{B}\subseteq \mathcal{A},\]
\[\mathcal{SL}\subseteq\mathcal{LNB}\subseteq{LRB}\subseteq\mathcal{A},\]
\[\mathcal{D}\subseteq \mathcal{C}\subseteq \mathcal{A}.\]
Since $\mathcal{A}\subseteq \mathcal{LLS}$, all of the above varieties are in the lattice of all left legal semigroup varieties.

\medskip
\noindent
It is known (see, for example, \cite{Evans} and \cite{Petrich}) that
\[\mathcal{B}=\mathcal {ZM}\vee \mathcal {LZ},\]
\[\mathcal{C}=\mathcal{B}\vee \mathcal{SL}=\mathcal {B}\vee \mathcal {LNB}=\mathcal{B}\vee \mathcal{D}=\mathcal{D}\vee \mathcal{LZ}=\mathcal{LNB}\vee\mathcal{ZM},\]
\[\mathcal{D}=\mathcal{ZM}\vee \mathcal{SL},\]
\[\mathcal{LNB}=\mathcal{LZ}\vee \mathcal{SL},\]
\[\mathcal {B}\wedge \mathcal {LNB}=\mathcal{LZ},\ \mathcal{C}\wedge \mathcal{LRB}=\mathcal{LNB}, \ \mathcal{D}\wedge \mathcal{LNB}=\mathcal{SL}.\]

\begin{theorem}\label{thmsup} For the varieties $\mathcal{A}$, $\mathcal{B}$, $\mathcal {C}$, $\mathcal{D}$, $\mathcal{LRB}$, and $\mathcal{ZM}$, the following equations hold:
$\mathcal {A}=\mathcal {ZM}\vee \mathcal {LRB}=\mathcal{D}\vee \mathcal{LRB}=\mathcal{C}\vee \mathcal{LRB}=\mathcal{B}\vee \mathcal{LRB}$.
\end{theorem}

\begin{proof}

\noindent
$\mathcal {A}=\mathcal {ZM}\vee \mathcal {LRB}$: Since the varieties $\mathcal {ZM}$ and $\mathcal {LRB}$ are subvarieties of the variety $\mathcal A$, we have $\mathcal {ZM}\vee \mathcal {LRB}\subseteq \mathcal {A}$. By Theorem~\ref{thm1}, every semigroup belonging to $\mathcal A$ is a subdirect product of a left regular band and a zero semigroup, and hence it belongs to $\mathcal {ZM}\vee \mathcal {LRB}$. Thus $\mathcal{A}\subseteq \mathcal {ZM}\vee \mathcal {LRB}$. Consequently $\mathcal {A}=\mathcal {ZM}\vee \mathcal {LRB}$.

\medskip

\noindent
$\mathcal {A}=\mathcal{D}\vee \mathcal{LRB}$: Using equations $\mathcal{D}=\mathcal{ZM}\vee \mathcal{SL}$,  $\mathcal {A}=\mathcal {ZM}\vee \mathcal {LRB}$, and the fact that $\mathcal{LS}\subseteq \mathcal{LRB}$, we have
\[\mathcal{D}\vee\mathcal{LRB}=\mathcal{ZM}\vee \mathcal{SL}\vee \mathcal{LRB}
=\mathcal{ZM}\vee \mathcal {LRB}=\mathcal{A}.\]

\medskip

\noindent
$\mathcal {A}=\mathcal{C}\vee \mathcal{LRB}$: Using equations $\mathcal{C}=\mathcal{B}\vee \mathcal{D}$, $\mathcal {B}=\mathcal {ZM}\vee \mathcal {LZ}$, $\mathcal{D}=\mathcal{ZM}\vee \mathcal{SL}$, $\mathcal {A}=\mathcal {ZM}\vee \mathcal {LRB}$, $\mathcal{LNB}=\mathcal{LZ}\vee \mathcal{SL}$, we get
\[\mathcal{C}\vee\mathcal{LRB}=\mathcal{B}\vee\mathcal{D}\vee\mathcal{LRB}=
\mathcal{LZ}\vee\mathcal{ZM}\vee \mathcal{ZM}\vee\mathcal{SL}\vee\mathcal{LRB}=\]
\[=\mathcal{ZM}\vee\mathcal{LRB}\vee\mathcal{LZ}\vee\mathcal{SL}=\mathcal{A}\vee\mathcal{LNB}=\mathcal{A},\]
because $\mathcal{LNB}\subseteq \mathcal{A}$.

\medskip

\noindent
$\mathcal {A}=\mathcal{B}\vee \mathcal{LRB}$: Using equations $\mathcal {B}=\mathcal {ZM}\vee \mathcal {LZ}$ and $\mathcal {A}=\mathcal{ZM}\vee \mathcal{LRB}$, we get
\[\mathcal{B}\vee\mathcal{LRB}=\mathcal{LZ}\vee\mathcal{ZM}\vee\mathcal{LRB}=\mathcal{LZ}\vee\mathcal{A}=\mathcal{A},\]
because $\mathcal{LZ}\subseteq\mathcal{A}$.
\end{proof}

\medskip

\noindent
An element $a$ of a lattice $(L;\vee, \wedge)$ with a zero $0$ is said to be an \emph{atom} of $L$ if $0<a$ and $0\leq b<a$ implies $0=b$ for arbitrary $b\in L$.
An element $a$ of a lattice $(L;\vee, \wedge)$ is called a \emph{neutral element} \cite{Birkhoff} of $L$ if,
for all $x, y\in L$, the triple $\{a, x, y\}$ generates a distributive sublattice of $L$. It is known \cite{Gratzer} that an element $a$ of a lattice $L$ is a neutral element if and only if, for all $x, y\in L$,
$(a\wedge x)\vee (a\wedge y)\vee (x\wedge y)=(a\vee x)\wedge (a\vee y)\wedge (x\vee y)$.

\begin{theorem} The varieties $\mathcal {ZM}$, $\mathcal {LZ}$, $\mathcal {SL}$ are atoms, the varieties $\mathcal T$, $\mathcal {ZM}$, $\mathcal {SL}$, $\mathcal{D}=\mathcal {ZM}\vee \mathcal{SL}$, and $\mathcal{LLS}$ are neutral elements of the lattice ${\mathcal L}({\bf LLS})$.
\end{theorem}

\begin{proof} By \cite[Theorem 2.6]{Kalicki}, $\mathcal {ZM}$, $\mathcal{LZ}$ and $\mathcal {SL}$ are atoms of the lattice of all semigroup varieties (see also the Theorem of Section IV of \cite{Evans}). Thus they are atoms of the lattice ${\mathcal L}({\bf LLS})$. It is clear that the varieties $\mathcal{T}$ and $\mathcal {LLS}$ are neutral elements of the lattice ${\mathcal L}({\bf LLS})$. By \cite[Proposition 4.1]{Volkov}, varieties $\mathcal {ZM}$, $\mathcal {SL}$, and $\mathcal{D}=\mathcal {ZM}\vee \mathcal {SL}$ are neutral elements of the lattice of all semigroup varieties. Thus they are neutral elements of the lattice ${\mathcal L}({\bf LLS})$.
\end{proof}

\noindent
The previous results on the lattice ${\mathcal L}({\bf LLS})$ is illustrated by Fig~\ref{Figure 1}.
\begin{figure}[htbp]
\begin{center}

\tikz{ \draw (0,1) node [left] {$\mathcal{{LZ}}$} node {$\bullet$} -- (2,0) node [below] {$\mathcal{{T}}$} node {$\bullet$} (2,0) -- (2,1.2) node [below right] {$\mathcal{{ZM}}$} node {$\bullet$} (2,0) -- (4,1) node [right] {$\mathcal{{SL}}$} node {$\bullet$} (0,1) -- (0,2.2) node [left] {$\mathcal{B}$} node {$\bullet$} (4,1) -- (4,2) node [right] {$\mathcal{D}$} node {$\bullet$} (4,2) -- (2,3.2) node [above left] {$\mathcal{C}$} node {$\bullet$}
(2,1.2) -- (4,2)  (0,1) -- (2,2) node [left] {$\mathcal{LNB}$} node {$\bullet$} (2,2) -- (4,3) node [right] {$\mathcal{LRB}$} node {$\bullet$} (2,2) -- (4,1) (3,4.2) -- (3,5.2) node [above] {$\mathcal{LLS}$} (0,2.2) -- (2,1.2) (2,3.2) -- (0,2.2) (3,4.2) -- (2,3.2)  (3,4.2)  node [right] {$\mathcal{A}$} node {$\bullet$} -- (4,3) (3,4.2) -- (2,1.2) (2,3.2) -- (2,2);}
\caption{}\label{Figure 1}
\end{center}
\end{figure}
The varieties on the lower level of the diagram are the left regular band varieties (see \cite[Figure 4]{Evans}). Their number is finite. However, the question of how many elements are in the lattice of all left legal semigroup varieties has not yet been answered. For example, the semigroup $S$ defined in Example~\ref{ex3a} is a left legal semigroup, but it is not included in any of the proper subvarieties of $\mathcal{LLS}$ on Fig. 1. The following problem needs a solution.

\bigskip
\noindent
{\bf OPEN PROBLEM:} Is the lattice of all left legal semigroup varieties finite?

\end{document}